\documentclass{amsart}
\usepackage{amssymb}
\usepackage{latexsym}
\usepackage{amsmath}
\usepackage{euscript} 
\usepackage{amsmath,amssymb,amsthm,graphicx,epstopdf,mathrsfs,url, color}
\usepackage{xcolor}
\usepackage{hyperref} 

  \def\dC{{\mathbb C}}

 \def\dN{{\mathbb N}} 
  \def\dR{{\mathbb R}}

\def\bm\chi{\mbox{\boldmath$\chi$}}

\def\ker{{\rm ker\,}}
\def\ran{{\rm ran\,}}

\def\dom{{\rm dom\,}}

\let\xker=\ker \def\ker{{\xker\,}}

\def\senki{{\lbrack\negthinspace [\bot ]\negthinspace\rbrack}}
\def\senki+{{\lbrack\negthinspace [+] \negthinspace\rbrack}}

\newtheorem{theorem}{Theorem}[section]
\newtheorem*{theorem*}{Theorem}
\newtheorem{proposition}[theorem]{Proposition}

\newtheorem{lemma}[theorem]{Lemma}
\theoremstyle{definition}

\numberwithin{equation}{section}

\begin{document}

\title{Schr\"odinger operators with oblique transmission conditions in $\dR^2$}

\author[J.~Behrndt]{Jussi~Behrndt}
\author[M.~Holzmann]{Markus Holzmann}
\author[G.~Stenzel]{Georg Stenzel}

\address{Institut f\"ur Angewandte Mathematik\\
Technische Universit\"at Graz \\
Steyrergasse 30\\
8010 Graz \\
Austria}
\email{behrndt@tugraz.at, holzmann@math.tugraz.at, gstenzel@math.tugraz.at}

\begin{abstract}
In this paper we study the spectrum of self-adjoint Schr\"odinger operators in $L^2(\mathbb{R}^2)$ with a new type of transmission conditions along a smooth closed curve $\Sigma\subseteq \mathbb{R}^2$.
Although these \textit{oblique} transmission conditions are formally similar to $\delta'$-conditions on $\Sigma$ (instead of the normal derivative here the Wirtinger derivative
is used) the spectral properties are significantly different: it turns out that for attractive interaction strengths the discrete spectrum is always unbounded from below.
Besides this unexpected spectral effect we also identify the essential spectrum, and we prove a Krein-type resolvent formula and a Birman-Schwinger principle. Furthermore, we show that these Schr\"odinger operators with oblique transmission conditions arise naturally as non-relativistic
limits of Dirac operators with electrostatic and Lorentz scalar $\delta$-interactions justifying their usage as models in quantum mechanics. 
\end{abstract}

\maketitle

\section{Introduction} \label{Introduction}
In many quantum mechanical applications one considers particles moving in an external potential field which
is localized near a set $\Sigma$ of measure zero.
Such strongly localized fields can be modeled by singular potentials that are supported on $\Sigma$ only;
of particular importance in this regard are $\delta$ and $\delta'$-interactions. To be more precise, assume
that $\Sigma$ splits $\dR^2$ into a bounded domain $\Omega_{+}$ and an unbounded domain $\Omega_{-} = \dR^2 \setminus \overline{\Omega_{+}}$,
and consider the formal Schr\"odinger differential expressions
\begin{equation} \label{FormalDiffExpressionDeltaDeltaP}
\mathcal{H}_{\delta, \alpha} = - \Delta + \alpha \delta_{\Sigma} \quad \text{and} \quad \mathcal{H}_{\delta', \alpha} = - \Delta + \alpha \delta'_{\Sigma},\quad
\alpha\in\dR.
\end{equation}
These singular perturbations of the free Schr\"odinger operator $-\Delta$ are characterized by certain transmission conditions
along the interface $\Sigma$
for the functions in the operator domain. For $\delta$-interactions one considers functions  $f : \dR^2 \to \dC$ such that
the restrictions $f_{\pm} = f \upharpoonright \Omega_{\pm}$ satisfy the transmission conditions
\begin{equation} \label{DInteraction}
f_{+} = f_{-} \quad \text{and} \quad -\frac{\alpha}{2} \big( f_{+} + f_{-} \big) = \big( \partial_{\nu} f_{+} - \partial_{\nu} f_{-} \big) \quad \text{on } \Sigma,
\end{equation}
 while $\delta'$-interactions are modeled by the transmission conditions
\begin{equation} \label{DprimeInteraction}
f_{+} - f_{-} = -\frac{\alpha}{2 } \big( \partial_{\nu} f_{+} + \partial_{\nu} f_{-} \big) \quad \text{and} \quad  \partial_{\nu} f_{+} = \partial_{\nu} f_{-} \quad \text{on } \Sigma;
\end{equation}
here $\partial_{\nu} f_{\pm}$ is the normal derivative and $\nu = (\nu_1, \nu_2)$ the unit normal vector field on $\Sigma$
pointing outwards of $\Omega_{+}$. The spectra and resonances of the self-adjoint realizations associated with the formal expressions \eqref{FormalDiffExpressionDeltaDeltaP} in $L^2(\dR^2)$ are well
understood, see, e.g., \cite{BLL13, BEKS94, E08, EK15, G16,G19a,G19b, GS14, MPS16}. In particular, the essential spectrum is given by $[0, \infty)$ and the discrete spectrum consists of at most finitely many points for every interaction strength $\alpha < 0$, while there is no negative spectrum if $\alpha \geq 0$.

In contrast to the transmission conditions \eqref{DInteraction} and \eqref{DprimeInteraction} we are interested in a new type of transmission conditions  of the form
\begin{equation}\label{ObliqueTC}
(\nu_1 + i \nu_2) \big( f_{+} - f_{-} \big)  = - \alpha\big( \partial_{\overline{z}} f_{+} + \partial_{\overline{z}} f_{-} \big) \quad \text{and} \quad  \partial_{\overline{z}} f_{+} = \partial_{\overline{z}} f_{-} \quad \text{on } \Sigma,
\end{equation}
where $\alpha\in\dR$ and $\partial_{\overline{z}} = \frac{1}{2} ( \partial_1 + i \partial_2)$ is the Wirtinger derivative. In the sequel such jump
conditions will be referred to as
\textit{oblique} transmission conditions.  Note that the conditions \eqref{ObliqueTC} can be rewritten as
\begin{equation} \label{OTCdNdT}
f_{+} - f_{-} = - \frac{\alpha}{2} \big( \partial_{\nu} f_{+} + \partial_{\nu}  f_{-} + i \partial_{t} f_{+} +  i \partial_{t} f_{-} \big)
\quad \text{and} \quad  \partial_{\overline{z}} f_{+} = \partial_{\overline{z}} f_{-} \quad \text{on } \Sigma,
\end{equation}
where  $\partial_t$ denotes the tangential derivative. Thus,
on a formal level there is some analogy to the $\delta'$-transmission conditions in \eqref{DprimeInteraction},
but it will turn out that the properties of the corresponding self-adjoint realization in $L^2(\dR^2)$
differ significantly from those of Schr\"odinger operators with $\delta'$-interactions.

To make matters mathematically rigorous, assume that the curve $\Sigma$ is the boundary of a bounded and simply connected $C^\infty$-domain $\Omega_+$ with open complement $\Omega_{-} = \dR^2 \setminus \overline{\Omega_{+}}$,
denote the $L^2$-based Sobolev space of first order by $H^1$, let
$\gamma_D^\pm:H^1(\Omega_{\pm})\rightarrow L^2(\Sigma)$ be
 the Dirichlet trace operators, and define for $\alpha\in\dR$ the Schr\"odinger operator with oblique transmission
conditions by
\begin{equation} \label{DefSchroedOpOBT}
\begin{split}
T_{\alpha} f &= \left( - \Delta f_{+} \right) \oplus  \left( - \Delta f_{-} \right),\\
\dom T_{\alpha}  &= \Big\{ f \in H^1(\Omega_{+}) \oplus H^1(\Omega_{-}) \: \big| \: \partial_{\overline{z}}f_{+} \oplus \partial_{\overline{z}} f_{-} \in H^1(\dR^2),  \\
&\hspace{1.5cm}  (\nu_1 + i \nu_2) \bigl(\gamma_D^+ f_{+} - \gamma_D^- f_{-} \bigr) = -  \alpha \bigl( \gamma_D^+ ( \partial_{\overline{z}} f_{+}) + \gamma_D^- ( \partial_{\overline{z}} f_{-} )\bigr) \Big\}.
\end{split}
\end{equation}
The next theorem is the main result in this paper. We discuss the spectral properties of the Schr\"odinger
operators $T_\alpha$ and, in particular, we show in item (ii) that for every $\alpha<0$ the operator $T_\alpha$ is necessarily unbounded from below and the discrete
spectrum in $(-\infty,0)$ is infinite and accumulates to $-\infty$. In items (iii) and (iv) we shall make use of the potential operator 
$\Psi_{\lambda}:L^2(\Sigma) \to L^2(\dR^2)$ and the single layer boundary integral operator $S(\lambda):L^2(\Sigma) \to L^2(\Sigma)$ 
defined in \eqref{DefPsi} and \eqref{DefSLBIO}, respectively.

\begin{theorem}  \label{ThmPropertiesTAlpha}
For any  $\alpha \in \dR$ the operator $T_{\alpha}$ is
self-adjoint in $L^2(\dR^2)$ and  the essential spectrum is given by
$$\sigma_{\text{\rm ess}}(T_{\alpha}) = [0, \infty).$$
Furthermore, the following statements hold:
\begin{enumerate}
\item[(i)] If $\alpha \geq 0$,  then  $\sigma_{\text{\rm disc}}(T_{\alpha}) = \emptyset$ and $T_\alpha$ is a nonnegative operator
in $L^2(\dR^2)$.
\item[(ii)] If $\alpha < 0$, then $\sigma_{\text{\rm disc}}(T_{\alpha})$ is infinite, unbounded from below, and does not accumulate to $0$. Moreover, for every fixed $n \in \mathbb{N}$ the $n$-th discrete eigenvalue $\lambda_n \in \sigma_{\text{\rm disc}}(T_{\alpha})$ 
(ordered non-increasingly) admits 
the asymptotic expansion
\begin{equation*}
\lambda_n = - \frac{4}{\alpha^2} + \mathcal{O}(1) \quad \text{for } \alpha \to 0^{-},
\end{equation*}
where the dependence on $n$ appears in the $\mathcal{O}(1)$-term. 
\item[(iii)] For $\lambda \in \dC \setminus [0 , \infty)$ the Birman-Schwinger principle is valid:
\begin{equation*} 
\lambda \in \sigma_{\rm p}(T_{\alpha}) \quad \Longleftrightarrow \quad 1 \in \sigma_{\rm p} \big(\alpha \lambda S(\lambda) \big).
\end{equation*}
\item[(iv)] For $\lambda \in \rho(T_{\alpha})=\dC \setminus ( [0, \infty) \cup \sigma_{\rm p}(T_{\alpha}))$ the operator $I - \alpha \lambda S(\lambda)$ is boundedly invertible in $L^2(\Sigma)$ and
the resolvent formula
\begin{equation*}
(T_{\alpha} - \lambda)^{-1} = (- \Delta - \lambda)^{-1} +\alpha \Psi_{\lambda} \bigl( I - \alpha \lambda S(\lambda) \bigr)^{-1} \Psi_{\overline{\lambda}}^{\ast}
\end{equation*}
holds,
where $- \Delta$ is the free Schr\"odinger operator defined on $H^2(\dR^2)$.
\end{enumerate}
\end{theorem}

To illustrate the significance of Theorem~\ref{ThmPropertiesTAlpha} we show that Schr\"odinger operators with oblique transmission conditions
arise naturally as non-relativistic limits of Dirac operators with electrostatic and Lorentz scalar $\delta$-interactions. 
To motivate this, consider one-dimensional Dirac operators with $\delta'$-interactions of strength $\alpha \in \mathbb{R}$ supported in $\Sigma=\{0\}$. These are first order differential operators in $L^2(\mathbb{R})^2$ and the singular interaction is modeled by transmission conditions for functions in the operator domain, which for sufficiently smooth $f = (f_1,f_2) \in L^2(\mathbb{R})^2$ are given by
\begin{equation} \label{transmission_conditions_Dirac_delta_prime}
f_1(0+)-f_1(0-)  = i \frac{\alpha c}{2} \big( f_2(0+)+f_2(0-) \big)  \quad \text{and} \quad f_2(0+) = f_2(0-),
\end{equation}
where $c>0$ is the speed of light. It is known that the associated self-adjoint Dirac operators converge in the non-relativistic limit to a Schr\"odinger operator with a $\delta'$-interaction of strength $\alpha$; cf. \cite{AGHH05, GS87} and also \cite{BMP17, CMP13} for generalizations. It is not difficult to see that~\eqref{transmission_conditions_Dirac_delta_prime} can be rewritten as the transmission conditions associated with a Dirac operator with a combination of an electrostatic and a Lorentz scalar $\delta$-interaction of strengths $\eta =- \frac{\alpha c^2}{2}$ and $\tau = \frac{\alpha c^2}{2}$, respectively, as they were studied in dimension one recently in \cite{BHSS22} and in higher space dimensions in, e.g., \cite{AMV15, BEHL19, BHOP20, BHSS22}.

To find a counterpart of the above result in dimension two, consider a Dirac operator with electrostatic and Lorentz scalar $\delta$-shell interactions of strength $\eta$ and $\tau$, respectively, supported on $\Sigma$, which is formally given by
\begin{equation} \label{def_A_eta_tau_formal}
\mathcal{A}_{\eta , \tau} = A_0 + \left( \eta I_2 + \tau \sigma_3 \right) \delta_{\Sigma};
\end{equation}
here $A_0$ is the unperturbed Dirac operator, $I_2$ is the $2 \times 2$-identity matrix and $\sigma_3 \in \mathbb{C}^{2 \times 2}$ is given in~\eqref{PauliMatrices}. The differential expression $\mathcal{A}_{\eta , \tau}$ gives rise to a self-adjoint operator $A_{\eta, \tau}$ in $L^2(\mathbb{R}^2)^2$, see~\eqref{DefAetaTau}. If one chooses, as above, $\eta = -\frac{\alpha c^2}{2}$ and $\tau = \frac{\alpha c^2}{2}$ and computes the non-relativistic limit, then instead of a Schr\"odinger operator with a $\delta'$-interaction one gets the somewhat unexpected limit $T_\alpha$. Of course, this is compatible with the one-dimensional result described above, as the one-dimensional counterparts of~\eqref{DprimeInteraction} and~\eqref{OTCdNdT} coincide, since there are no tangential derivatives in $\mathbb{R}$. However, in higher dimensions Schr\"odinger operators with oblique transmission conditions should be viewed as the non-relativistic counterparts of Dirac operators with transmission conditions generalizing~\eqref{transmission_conditions_Dirac_delta_prime}.
Related results on nonrelativistic limits of three-dimensional Dirac operators with singular interactions can be found in \cite{BEHL18, BEHL19, H19}.
The precise result about the non-relativistic limit described above is stated in the following theorem and shown in Section~\ref{SectionNonrelLimit}.

\begin{theorem} \label{ThmNonRelLimit}
Let $\alpha \in \dR$. 
Then for all $\lambda \in \mathbb{C} \setminus \mathbb{R}$ one has  
\begin{equation*}
\lim_{c \rightarrow \infty} \left( A_{-\alpha c^2/2, \alpha c^2/2} - (\lambda + c^2 /2) \right)^{-1} = \left( \begin{array}{cc}
\left( T_{\alpha} - \lambda \right)^{-1} & 0 \\
0 & 0 
\end{array} \right),
\end{equation*}
where the convergence is in the operator norm and the convergence rate is $\mathcal{O}\left( \frac{1}{c} \right)$.
\end{theorem}

\textbf{Notations.}  Throughout this paper $\Omega_{+} \subseteq \dR^2$ is a bounded and simply connected $C^{\infty}$-domain 
and $\Omega_{-} = \dR^2 \setminus \overline{\Omega_{+}}$ is the corresponding exterior domain with 
boundary $\Sigma = \partial \Omega_{-} = \partial \Omega_{+}$. 
The unit normal vector field on $\Sigma$ pointing outwards of $\Omega_+$ is denoted by $\nu$. Moreover, for $z \in \dC \setminus[0, \infty)$ we choose the square root  $\sqrt{z}$ such that  $\text{Im}  \sqrt{z}  > 0$ holds. The modified Bessel function of order $j \in \dN_0$ is denoted by $K_j$. 
 
For $s \geq 0$ the spaces $H^s(\dR^2)^n$, $H^s(\Omega_{\pm})^n$, and $H^s(\Sigma)^n$ are the standard $L^2$-based Sobolev spaces of $\dC^n$-valued functions defined on $\dR^2$, $\Omega_{\pm}$, and $\Sigma$, respectively. If $n=1$ we simply write $H^s(\dR^2)$, $H^s(\Omega_{\pm})$, and $H^s(\Sigma)$. For negative $s < 0$ we define the spaces 
$H^s(\dR^2)^n$ and 
$H^s(\Sigma)^n$ as the anti-dual spaces of $H^{-s}(\dR^2)^n$ and $H^{-s}(\Sigma)^n$, respectively. 
We denote the restrictions of functions $f: \mathbb{R}^2 \rightarrow \mathbb{C}^n$ onto $\Omega_\pm$ by $f_\pm$; in this sense we write $H^1(\mathbb{R}^2 \setminus \Sigma)^n = H^1(\Omega_+)^n \oplus H^1(\Omega_-)^n$ and identify $f \in H^1(\mathbb{R}^2 \setminus \Sigma)^n$ with $f_+ \oplus f_-$, where $f_\pm \in H^1(\Omega_\pm)^n$.
The Dirichlet trace operators are denoted by $\gamma_D^\pm:H^1(\Omega_{\pm})\rightarrow L^2(\Sigma)$ 
and we shall write $\gamma_D:H^1(\mathbb R^2)\rightarrow L^2(\Sigma)$ for the trace on $H^1(\mathbb R^2)$; sometimes 
these trace operators are also viewed as bounded mappings to $H^{1/2}(\Sigma)$.

For a Hilbert space $\mathcal{H}$ we write $\mathcal{L}(\mathcal{H})$ for the space of all everywhere defined, linear, and bounded operators on $\mathcal{H}$.  Furthermore, the domain, kernel, and range of a linear operator $T$ from a Hilbert space $\mathcal{G}$ to $\mathcal{H}$ are denoted by $\text{dom}\, T$, $\text{ker}\, T$, and $\text{ran}\, T$, respectively. The resolvent set, the spectrum, the essential spectrum, the discrete spectrum, and the point spectrum of a self-adjoint operator $T$ are denoted by $\rho(T)$, $\sigma(T)$, $\sigma_{\text{ess}}(T)$, $\sigma_{\text{disc}}(T)$, and $\sigma_\textup{p}(T)$.  The eigenvalues of compact self-adjoint operators $K \in \mathcal{L}(\mathcal{H})$ are denoted by $\mu_n(K)$ and are ordered by their absolute values.

\subsection*{Acknowledgement.}
We are indebted to the referee for a very careful reading of our manuscript and various helpful suggestions to improve the text.
Jussi Behrndt and Markus Holzmann gratefully acknowledge financial support by the Austrian Science Fund (FWF): P33568-N. This publication is based upon work from COST Action CA 18232 MAT-DYN-NET, supported by COST (European Cooperation in Science and Technology), www.cost.eu.

\section{Proof of Theorem~\ref{ThmPropertiesTAlpha}} \label{SectionSchroedOpsobligetrans}

In this section the main result of this paper will be proved. For this, some families of integral operators are used. Define for $\lambda \in \dC \setminus [0, \infty)$ the function $L_{\lambda}$ by
\begin{equation} \label{DefIntegralkernPhiLambda}
L_{\lambda}(x) = \frac{\sqrt{ \lambda}}{2 \pi} K_1 \bigl( - i \sqrt{ \lambda} |x| \bigr) \frac{ x_1 - i x_2}{|x|}, \quad  x = (x_1,x_2) \in \mathbb{R}^2 \setminus \{0\},
\end{equation} 
and the operator $\Psi_{\lambda} : L^2(\Sigma) \to L^2(\dR^2)$ by
\begin{equation}  \label{DefPsi}
\Psi_{\lambda} \varphi (x) = \int_{\Sigma} L_{\lambda}(x-y) \varphi (y) d \sigma(y), \quad \varphi 
\in L^2(\Sigma),  ~x \in \dR^2\setminus\Sigma.
\end{equation}
Moreover, for $\lambda \in \dC \setminus [0, \infty)$ we make use of the single layer potential $SL(\lambda): L^2(\Sigma) \rightarrow H^1(\mathbb{R}^2)$ and the 
single layer boundary integral operator $S(\lambda): L^2(\Sigma) \to L^2(\Sigma)$ associated with $-\Delta - \lambda$ that are defined by
\begin{equation} \label{def_SL_potential}
  SL(\lambda) \varphi(x) = \int_\Sigma \frac{1}{2 \pi} K_0 \bigl( - i \sqrt{\lambda} |x- y| \bigr) \varphi(y) d \sigma(y), \quad \varphi \in L^2(\Sigma),\quad x \in \mathbb{R}^2 \setminus \Sigma,
\end{equation}
and
\begin{equation} \label{DefSLBIO}
S(\lambda) \varphi(x) = \int_{\Sigma} \frac{1}{2 \pi} K_0 \bigl( - i \sqrt{\lambda} |x- y| \bigr) \varphi(y) d \sigma(y), \quad \varphi \in L^2(\Sigma), ~x \in \Sigma.
\end{equation}
It is known that $SL(\lambda)$ and $S(\lambda)$ are bounded and $\text{ran}\, S(\lambda)  \subseteq H^{1}(\Sigma)$; cf. \cite[Theorem~6.12 and Theorem~7.2]{M00}. In particular, $S(\lambda)$ gives rise to a compact  operator in $H^s(\Sigma)$ for every $s \in [0,1]$. Furthermore, 
$S(\lambda)$ is self-adjoint and positive for $\lambda < 0$ (see \textit{Step~1} in the proof of Proposition~\ref{EigenvalueFuncUnbdd}).
Some 
properties of $\Psi_{\lambda}$ and $S(\lambda)$ that are important in the proof of Theorem~\ref{ThmPropertiesTAlpha} 
are summarized in the following two propositions; cf. Appendix~\ref{Appendix} for the proof of Proposition~\ref{PropertiesPsi}
and Proposition~\ref{EigenvalueFuncUnbdd}.

 \begin{proposition} \label{PropertiesPsi}
Let $\lambda \in \dC \setminus [0, \infty)$ and let $\Psi_{\lambda}$ be given by~\eqref{DefPsi}. Then
\begin{equation} \label{relation_Psi_lambda_SL}
  \Psi_{\lambda} = -2 i \partial_z SL(\lambda): L^2(\Sigma) \to L^2(\dR^2)
\end{equation}
is bounded and the following is true:
\begin{enumerate} 
\item[(i)] $\Psi_{\lambda}$  gives rise to a bijective mapping $\Psi_{\lambda} : H^{1/2}(\Sigma) \to \mathcal{H}_{\lambda}$, where 
\begin{multline*}
  \mathcal{H}_{\lambda} := \big\{ f \in H^1(\mathbb{R}^2 \setminus \Sigma)  \: | \: \partial_{\overline{z}}f_{+} \oplus \partial_{\overline{z}} f_{-} \in H^1(\dR^2),  \left( - \Delta - \lambda \right) f_{\pm} = 0 \text{ on } \Omega_{\pm} \}.
\end{multline*}
\item[(ii)] $\Psi_{\lambda}^{\ast} : L^2(\dR^2) \to L^2(\Sigma)$ is compact, $\Psi_\lambda^* = -2i \gamma_D \partial_{\overline{z}} (-\Delta-\overline{\lambda})^{-1}$, and $\ran \Psi_{\lambda}^{\ast} \subseteq H^{1/2}(\Sigma)$.
\item[(iii)] For all $\varphi \in H^{1/2}(\Sigma)$ the jump relations 
\begin{equation*}
\begin{split}
i ( \nu_1 + i \nu_2) \big(  \gamma_D^{+}  ( \Psi_{\lambda} \varphi )_{+} -   \gamma_D^{-}  ( \Psi_{\lambda} \varphi )_{-} \big) &= \varphi,
\\ 
-i \big(  \gamma_D^{+} \partial_{\overline{z}} ( \Psi_{\lambda} \varphi )_{+} +   \gamma_D^{-} \partial_{\overline{z}} ( \Psi_{\lambda} \varphi )_{-} \big) &= \lambda S(\lambda)  \varphi,
\end{split}
\end{equation*} 
hold.
\end{enumerate}
\end{proposition}

For $\lambda < 0$ denote by $\mu_n(S(\lambda))$ the discrete eigenvalues of the positive self-adjoint operator $S(\lambda)$ 
order non-increasingly and with multiplicities taken into account.

\begin{proposition} \label{EigenvalueFuncUnbdd}
Let $S(\lambda)$ be defined by~\eqref{DefSLBIO} and let $n \in \mathbb{N}$ be fixed.  Then the following  holds:
\begin{enumerate}
\item[(i)] The function $(- \infty, 0) \ni \lambda \mapsto \lambda \mu_n( S(\lambda) )$ is continuous, strictly monotonically increasing and 
\begin{equation*}
\lim_{\lambda \to 0^{-} }\lambda \mu_n( S(\lambda) ) = 0 \quad \text{and} \quad \lim_{\lambda \to -\infty} \lambda \mu_n( S(\lambda) ) = - \infty.
\end{equation*}
\item[(ii)] For $a<0$ the unique solution $\lambda_n(a) \in (- \infty, 0)$ of $\lambda \mu_n(S(\lambda)) = a$ 
(see {\rm (i)}) admits the asymptotic expansion $\lambda_n(a) = - 4 a^2 + \mathcal{O}(1)$  for $a \to -  \infty$, 
where the dependence on $n$ appears in the $\mathcal{O}(1)$-term.
\end{enumerate}
\end{proposition}

\begin{proof}[Proof of Theorem \ref{ThmPropertiesTAlpha}]
\emph{Step 1:}  We verify that $T_{\alpha}$ is symmetric in $L^2(\dR^2)$.
Observe first that for $f \in \dom T_\alpha$ we have $\partial_{\overline{z}} f_{\pm} \in H^1(\Omega_{\pm})$
and $\Delta f_\pm = 4 \partial_{z}  \partial_{\overline{z}}  f_{\pm}  \in L^2(\Omega_{\pm})$, and hence $T_\alpha$ is well-defined. Moreover, as $C_0^{\infty}(\mathbb{R}^2 \setminus \Sigma) \subseteq \dom T_{\alpha}$ it is also clear that 
$\text{dom}\,T_\alpha$ is dense.
In order to show that $T_{\alpha}$ is symmetric, we note that integration by parts in $\Omega_{\pm}$ yields for  $f,g \in \dom T_{\alpha}$
\begin{equation}  \label{TASymmetricIbP}
\begin{split}
( - \Delta f_{\pm} , g_{\pm} &)_{L^2(\Omega_{\pm})} = ( -4  \partial_{z}  \partial_{\overline{z}}  f_{\pm} , g_{\pm} )_{L^2(\Omega_{\pm})}  \\
&=  4 ( \partial_{\overline{z}}  f_{\pm} , \partial_{\overline{z}}  g_{\pm} )_{L^2(\Omega_{\pm})}  \mp  2 \big( (\nu_1 - i \nu_2)  \gamma_D^{\pm} ( \partial_{\overline{z}}  f_{\pm} ),  \gamma_D^{\pm} g_{\pm} \big)_{L^2(\Sigma)} \\
 &= 4  ( \partial_{\overline{z}}  f_{\pm} , \partial_{\overline{z}}  g_{\pm} )_{L^2(\Omega_{\pm})} \mp 2 \big(  \gamma_D^{\pm} ( \partial_{\overline{z}}  f_{\pm}),  (\nu_1 + i \nu_2)   \gamma_D^{\pm} g_{\pm}) \big)_{L^2(\Sigma)}.
\end{split}
\end{equation}
Now, consider \eqref{TASymmetricIbP} for $f=g$ and add the equations for $\Omega_{+}$ and $\Omega_{-}$. Then, using 
$\gamma_D^{+} ( \partial_{\overline{z}} f_{+})=\gamma_D^{-} ( \partial_{\overline{z}} f_{-})$ and
the transmission condition for $f \in \dom T_{\alpha}$, one finds that
\begin{equation} \label{FormForTA}
\begin{split}
\big( T_{\alpha} f , f& \big)_{L^2(\dR^2)} = 4 \big( \| \partial_{\overline{z}} f_{+} \|^2_{L^2(\Omega_{+})} +  \| \partial_{\overline{z}} f_{-} \|^2_{L^2(\Omega_{-})} \big)  \\
&\quad -   \big(  \gamma_D^{+} ( \partial_{\overline{z}} f_{+}) +  \gamma_D^{-} ( \partial_{\overline{z}} f_{-} ) , (\nu_1 + i \nu_2) (  \gamma_D^{+} f_{+} -  \gamma_D^{-} f_{-} ) \big)_{L^2(\Sigma)} \\
&= 4 \| \partial_{\overline{z}} f_{+} \oplus \partial_{\overline{z}} f_{-} \|^2_{L^2(\mathbb{R}^2)} + \alpha \|   \gamma_D^{+} ( \partial_{\overline{z}} f_{+}) +  \gamma_D^{-} ( \partial_{\overline{z}} f_{-})  \|^2_{L^2(\Sigma)} \in \mathbb{R}.
\end{split}
\end{equation}
Since this holds for all $f \in \dom T_{\alpha}$, we conclude that $T_{\alpha}$  is symmetric.

\emph{Step 2:}  Proof of the Birman-Schwinger principle in~(iii): To show the first implication, assume that $\lambda \in \dC \setminus [0, \infty)$ with  $1 \in \sigma_\textup{p} ( \alpha \lambda S(\lambda) )$ is given and choose $\varphi \in \ker( I - \alpha \lambda S(\lambda)) \setminus \{0\}$. Then it follows from the mapping properties of $S(\lambda)$ that $\varphi = \alpha  \lambda S(\lambda) \varphi \in H^{1/2}(\Sigma)$  holds. Therefore, Proposition \ref{PropertiesPsi}~(i) implies that $f := \Psi_{\lambda} \varphi \in \mathcal{H}_\lambda$ fulfils $f \neq 0$, $f \in H^1(\mathbb{R}^2 \setminus \Sigma)$,  $\partial_{\overline{z}}f_{+} \oplus \partial_{\overline{z}} f_{-} \in H^1(\dR^2)$ and, as $\varphi \in \ker( 1 - \alpha \lambda S(\lambda)) \setminus \{0\}$, Proposition~\ref{PropertiesPsi}~(iii) implies
\begin{equation*}
 i (\nu_1 + i \nu_2) \big( \gamma_D^{+} f_{+} -  \gamma_D^{-} f_{-} \big) = \varphi = \alpha \lambda S(\lambda) \varphi = -i \alpha \big(  \gamma_D^{+} ( \partial_{\overline{z}} f_{+}) +  \gamma_D^{-} ( \partial_{\overline{z}} f_{-} ) \big).
\end{equation*}
Hence, $f \in \dom T_\alpha$. Moreover, as $f \in \mathcal{H}_\lambda$, we conclude  $f \in \ker( T_{\alpha} - \lambda) \setminus \{0\}$ and hence $\lambda \in \sigma_\textup{p}(T_{\alpha})$.

To show the second implication, assume that $\lambda \in \sigma_\textup{p}(T_{\alpha})$ is given and choose $f \in \ker(T_{\alpha} - \lambda)  \setminus \{0\}$.  Then, by Proposition \ref{PropertiesPsi}~(i) there exists a unique $\varphi \in H^{1/2}(\Sigma)$ such that $f = \Psi_{\lambda} \varphi$. Moreover, using $f \in \dom T_\alpha$ and Proposition~\ref{PropertiesPsi}~(iii) one finds that
\begin{equation*}
0 =  i (\nu_1 + i \nu_2) \big( \gamma_D^{+} f_{+} -  \gamma_D^{-} f_{-} \big) + i \alpha \big( \gamma_D^{+} ( \partial_{\overline{z}} f_{+}) +  \gamma_D^{-} ( \partial_{\overline{z}} f_{-} ) \big) = ( I - \alpha \lambda S(\lambda))\varphi.
\end{equation*}
Since $\varphi \neq 0$, we conclude $1 \in \sigma_\textup{p} ( \alpha\lambda S(\lambda))$.

\emph{Step 3:} Next, we prove that $T_{\alpha}$ is a self-adjoint operator and the resolvent formula in~(iv). Let $\lambda \in \dC \setminus (  [0, \infty) \cup \sigma_\textup{p}(T_{\alpha}) )$ be fixed. First, we show that  $I - \alpha \lambda S(\lambda)$ gives rise to a bijective map in $H^s(\Sigma)$  for every $s \in [0,1]$. Recall that $S(\lambda)$ is compact in $H^s(\Sigma)$. Since $I - \alpha \lambda S(\lambda)$ is injective for our choice of $\lambda$ by the Birman-Schwinger principle in~(iii), Fredholm's alternative shows that $I - \alpha \lambda S(\lambda)$ is indeed bijective.

Recall that $T_\alpha$ is symmetric; cf. \textit{Step~1}. Hence, to show that $T_\alpha$ is self-adjoint, it suffices to verify that $\text{ran}(T_\alpha-\lambda) = L^2(\mathbb{R}^2)$ holds for 
$\lambda \in \dC \setminus (  [0, \infty) \cup \sigma_\textup{p}(T_{\alpha}) )$. Fix such a $\lambda$, let $f \in L^2(\mathbb{R}^2)$, and define
\begin{equation} \label{ResolventofTA}
g = (- \Delta - \lambda)^{-1}f + \alpha \Psi_{\lambda} (I - \alpha \lambda S(\lambda))^{-1} \Psi_{\overline{\lambda}}^{\ast} f,
\end{equation}
which is well-defined by the considerations above.
Since $\Psi_{\overline{\lambda}}^{\ast} f \in H^{1/2}(\Sigma)$ by Proposition~\ref{PropertiesPsi}~(ii) and $(I - \alpha \lambda S(\lambda))^{-1}$ is bijective in $H^{1/2}(\Sigma)$, we conclude with Proposition~\ref{PropertiesPsi}~(i) that $\Psi_{\lambda} (I - \alpha \lambda S(\lambda))^{-1} \Psi_{\overline{\lambda}}^{\ast} f \in \mathcal{H}_\lambda  \subseteq H^1(\mathbb{R}^2 \setminus \Sigma)$. In particular, with $(- \Delta - \lambda)^{-1}f \in H^2(\mathbb{R}^2)$ this implies that $g \in H^1(\mathbb{R}^2 \setminus \Sigma)$ and $\partial_{\overline{z}}g_{+} \oplus \partial_{\overline{z}} g_{-} \in H^1(\dR^2)$. Moreover, with Proposition~\ref{PropertiesPsi}~(ii)--(iii) we obtain that
\begin{equation*}
\begin{split}
i (\nu_1 + i \nu_2) &\big(  \gamma_D^{+} g_{+} -  \gamma_D^{-} g_{-} \big) + i \alpha \big(  \gamma_D^{+} ( \partial_{\overline{z}}g_{+}) +  \gamma_D^{-}(\partial_{\overline{z}} g_{-}) \big) \\
&= \alpha  (I - \alpha \lambda S(\lambda) )^{-1} \Psi_{\overline{\lambda}}^{\ast} f  - \alpha \Psi_{\overline{\lambda}}^{\ast} f - \alpha^2 \lambda S(\lambda) (I - \alpha \lambda S(\lambda) )^{-1} \Psi_{\overline{\lambda}}^{\ast} f \\
&= \alpha  (I - \alpha \lambda S(\lambda))  (I - \alpha\lambda S(\lambda) )^{-1}  \Psi_{\overline{\lambda}}^{\ast} f  - \alpha \Psi_{\overline{\lambda}}^{\ast} f  = 0
\end{split}
\end{equation*}
and hence, $g \in \dom T_{\alpha}$.  As $\Psi_{\lambda} (I - \alpha \lambda S(\lambda))^{-1} \Psi_{\overline{\lambda}}^{\ast} f \in \mathcal{H}_\lambda$ by Proposition~\ref{PropertiesPsi}~(i), we conclude
\begin{equation*}
\begin{split}
\left( - \Delta - \lambda \right) g_{\pm} &= \left( -\Delta - \lambda \right) \big( (- \Delta - \lambda )^{-1} f \big)_{\pm} +  \alpha \left( -  \Delta - \lambda \right) \big(  \Psi_{\lambda} (I - \alpha \lambda S(\lambda) )^{-1} \Psi_{\overline{\lambda}}^{\ast} f \big)_{\pm}\\
&= \left( -  \Delta - \lambda \right) \big( (- \Delta - \lambda )^{-1} f \big)_{\pm} = f_{\pm},
\end{split}
\end{equation*}
i.e. $(T_{\alpha} - \lambda) g = f$. Since $f \in L^2(\dR^2)$ was arbitrary, we conclude that $\ran (T_{\alpha} - \lambda) = L^2(\dR^2)$ and that $T_{\alpha}$ is self-adjoint. Moreover, the resolvent formula in item~(iv) follows from \eqref{ResolventofTA}.

\emph{Step 4:} Next, we show $\sigma_{\text{ess}}(T_{\alpha}) = [0, \infty)$. Let $\lambda \in \dC \setminus \mathbb{R}$. Since $\Psi_{\overline{\lambda}}^{\ast} : L^2(\dR^2) \to L^2(\Sigma)$  is compact by Proposition~\ref{PropertiesPsi}~(ii), the resolvent formula in~(iv) implies that $(T_{\alpha} - \lambda)^{-1} - (-\Delta - \lambda)^{-1}$ is a compact operator in $L^2(\dR^2)$. Consequently, Weyl's Theorem \cite[Theorem~XIII.14]{RS77} yields that $\sigma_\textup{ess}(T_\alpha) = \sigma_\textup{ess}(-\Delta)=[0,\infty)$.

\emph{Step 5:} Proof of~(i): Let $\alpha \geq 0$. Then,~\eqref{FormForTA} implies that $T_\alpha$ is non-negative and hence, $\sigma(T_\alpha) \subset [0,\infty)$. Since the latter set coincides with $\sigma_\textup{ess}(T_\alpha)$, see \textit{Step~4}, we conclude $\sigma_\textup{disc}(T_\alpha)=\emptyset$.

\emph{Step 6:} Proof of~(ii): Let  $\alpha<0$. Since $\sigma_{\text{ess}}(T_{\alpha}) = [0, \infty)$, it follows from the Birman-Schwinger principle in~(iii) that 
\begin{equation*}
\sigma_{\text{disc}}(T_{\alpha}) = \left\{ \lambda_n \: | \:   n \in \dN \right\} =   \left\{ \lambda < 0 \: | \:  \exists n \in \dN \text{ such that } \lambda \mu_n(S(\lambda)) = \alpha^{-1} \right\}
\end{equation*} 
holds. Note that by Proposition~\ref{EigenvalueFuncUnbdd} the equation $\lambda \mu_n(S(\lambda)) = \alpha^{-1}$ has a unique solution $\lambda_n$ for all $n \in \mathbb{N}$. Moreover, for any $n \in \dN$  there cannot be infinitely many  $k \neq n$ with $\lambda_n = \lambda_k$, since otherwise $\alpha^{-1} < 0$ would be an eigenvalue with infinite multiplicity of the self-adjoint and compact operator $\lambda_n S(\lambda_n)$. Thus $\sigma_{\text{disc}}(T_{\alpha})$ is indeed an infinite set.  Furthermore, as $S(\lambda)$ is a positive self-adjoint operator in $L^2(\Sigma)$; cf. \textit{Step~1} in the proof of Proposition~\ref{EigenvalueFuncUnbdd}, we have by definition $\mu_n(S(\lambda)) \geq \mu_{n+1}(S(\lambda))$ implying $\lambda \mu_n ( S(\lambda)) \leq \lambda \mu_{n+1}(S(\lambda))$. Therefore, the monotonicity of the map $\lambda \mapsto \lambda \mu_n (S(\lambda))$ from Proposition \ref{EigenvalueFuncUnbdd} yields $\lambda_{n+1} \leq \lambda_{n}$ for all $n \in \dN$. This shows that $0$ cannot be an accumulation point of the sequence $(\lambda_n)_{n \in \dN}$ and 
as $\sigma_{\text{ess}}(T_{\alpha})\cap (-\infty,0)=\emptyset$ the sequence $(\lambda_n)_{n \in \dN}$ has no finite
accumulation points, that is, 
$\sigma_{\text{disc}}(T_{\alpha})$ must be unbounded from below. 

It remains to prove the asymptotic expansion in item~(ii). By the above considerations $\lambda_n$ is determined as the unique solution of $\lambda \mu_n(S(\lambda)) = \alpha^{-1}$. Clearly, if $\alpha \rightarrow 0^-$, then $a := \alpha^{-1} \rightarrow -\infty$. Hence, it follows from Proposition~\ref{EigenvalueFuncUnbdd}~(ii) with $a = \alpha^{-1}$ that 
$\lambda_n = - \frac{4}{\alpha^2} + \mathcal{O}(1)$ for $\alpha \rightarrow 0^-$
and that the dependence on $n$ appears in the $\mathcal{O}(1)$-term.
\end{proof}

\section{Proof of Theorem~\ref{ThmNonRelLimit}} \label{SectionNonrelLimit}

In this section we show that $T_\alpha$ is the nonrelativistic limit of a family of Dirac operators with electrostatic and Lorentz scalar $\delta$-shell potentials formally given by~\eqref{def_A_eta_tau_formal}, whose interaction strengths are suitably scaled. 
First, we recall the rigorous definition of the operator $A_{\eta, \tau}$ associated with~\eqref{def_A_eta_tau_formal}, see \cite{BEHL19, BHOP20, BHSS22} for details.
Let
\begin{equation} \label{PauliMatrices}
\sigma_1 = \left( \begin{array}{cc}
0 & 1\\                                              
1 & 0 \\                                            
\end{array}\right), \hspace{3mm} \hspace{3mm} \sigma_2 = \left( \begin{array}{cc}
0 & -i\\                                              
i & 0 \\                                            
\end{array}\right), \hspace{3mm} \text{ and } \hspace{3mm} \sigma_3 = \left( \begin{array}{cc}
1 & 0\\                                              
0 & -1 \\                                            
\end{array}\right),
\end{equation}
be the Pauli spin matrices and denote the $2 \times 2$ identity matrix by $I_2$.
Furthermore, for $x = (x_1,x_2) \in \dR^2$ we will use the abbreviations  
\begin{equation} \label{ShortNotation}
\sigma \cdot x = \sigma_1 x_1 + \sigma_2 x_2 \quad \text{and} \quad \sigma \cdot \nabla = \sigma_1 \partial_1 + \sigma_2 \partial_2.
\end{equation}
We define Dirac operators with electrostatic and Lorentz scalar $\delta$-shell interactions of strengths $\eta , \tau \in \dR$  in $L^2(\dR^2)^2$ by 
\begin{equation} \label{DefAetaTau}
  \begin{split}
A_{\eta , \tau} f &= \left( - i c (\sigma \cdot \nabla)  +\frac{c^2}{2} \sigma_3 \right)f_{+} \oplus \left(  - i c (\sigma \cdot \nabla)  + \frac{c^2}{2} \sigma_3 \right)f_{-}, \\
   \dom A_{\eta , \tau}  
   &=  \Big\{  f  \in H^1(\Omega_{+})^2 \oplus  H^1(\Omega_{-})^2 \: \big|   \\
 & ~~ i c (\sigma \cdot \nu) \left( \gamma_D^{+}f_{+} -  \gamma_D^{-}f_{-} \right) + \frac{1}{2} ( \eta I_2+ \tau \sigma_3 )  \left( \gamma_D^{+}f_{+} +  \gamma_D^{-}f_{-} \right) = 0 \Big\}.
  \end{split} 
\end{equation}
It is shown in \cite{BHOP20, BHSS22} that $A_{\eta, \tau}$ is self-adjoint in $L^2(\dR^2)^2$, whenever $\eta^2-\tau^2 \neq 4 c^2$, and as in  \cite{BEHL19} one sees that these operators are the self-adjoint realisations of the formal differential expression~\eqref{def_A_eta_tau_formal}.
In the above definition we are using units such that $\hbar = 1$ and consider the mass $m=\frac{1}{2}$, but we keep the speed of light $c$ as a parameter for the discussion of the non-relativistic limit $c \to \infty$.

Throughout this section we make use of the self-adjoint free Dirac operator $A_0$, which coincides with $A_{0,0}$ given in~\eqref{DefAetaTau} and which is defined on $H^1(\mathbb{R}^2)^2$. For $\lambda \in \rho(A_0)=\mathbb{C} \setminus ((-\infty, -\frac{c^2}{2}] \cup [\frac{c^2}{2}, \infty))$ 
the integral kernel of the resolvent of $A_0$ is given by $G_\lambda(x-y)$, where $G_\lambda(x)$ is defined for $x \in \dR^2 \setminus \{0\}$ by
\begin{multline} \label{IntegralkernelfreeDirac}
G_{\lambda}(x) = \frac{1}{2 \pi c} \sqrt{ \frac{\lambda^2}{c^2} - \frac{c^2}{4}} K_1 \left( -i  \sqrt{ \frac{\lambda^2}{c^2} - \frac{c^2}{4}} |x| \right) \frac{1}{|x|} (\sigma \cdot x) \\
+  \frac{1}{2 \pi c} K_0 \left( -i  \sqrt{ \frac{\lambda^2}{c^2} - \frac{c^2}{4} } |x| \right) \left( \frac{\lambda}{c} I_2 + \frac{c}{2} \sigma_3 \right);
\end{multline} 
cf. \cite[equation~(3.2)]{BHOP20}.
With this function we define the two families of integral operators 
\begin{equation} \label{PhiLandCL}
\begin{split}
\Phi_{\lambda} \varphi(x) &= \int_{\Sigma} G_{\lambda}(x-y) \varphi(y) \mathrm{d} \sigma(y), \quad \varphi \in L^2(\Sigma)^2, ~x \in \dR^2\setminus\Sigma, \\
\mathcal{C}_{\lambda} \varphi(x) &=  \lim_{\varepsilon \to 0^{+}} \int\limits_{\Sigma \setminus B(x,\varepsilon)} G_{\lambda}(x-y) \varphi(y) \mathrm{d} \sigma(y), \quad \varphi \in L^2(\Sigma)^2, ~ x \in \Sigma,
\end{split}
\end{equation}
where $B(x,\varepsilon)$ is the ball of radius $\varepsilon$ centered at $x$.
Both operators  $\Phi_{\lambda} : L^2(\Sigma)^2 \to L^2(\dR^2)^2$ and  $\mathcal{C}_{\lambda} :  L^2(\Sigma)^2 \to  L^2(\Sigma)^2$ are well-defined and bounded; 
cf. \cite[Proposition~3.3 and equation~(3.7)]{BHOP20}. 

In the following lemma, which is a preparation for the proof of Theorem~\ref{ThmNonRelLimit}, we will use the matrices
\begin{equation*}
M_1 = \left( \begin{array}{cc}
1 & 0 \\
0 & 0 
\end{array} \right), \quad M_2 = \left( \begin{array}{cc}
0 & 1 \\
0 & 0 
\end{array} \right), \quad \text{and} \quad M_3 = \left( \begin{array}{cc}
0 & 0 \\
0 & 1 
\end{array} \right);
\end{equation*}
products of scalar operators and matrices are understood componentwise, e.g. 
\begin{equation*}
(-\Delta - \lambda)^{-1} M_1  = \left( \begin{array}{cc}
(-\Delta - \lambda)^{-1} & 0 \\
0 & 0 
\end{array} \right): L^2(\dR^2)^2 \to L^2(\dR^2)^2.
\end{equation*}

\begin{lemma} \label{NonrelLimitProp}
Let $\lambda \in \dC \setminus \dR$. Then there exists a constant $K>0$, depending only on $\lambda$ and $\Sigma$, such that the estimates
\begin{subequations}
\begin{align}
	\label{estimate_A_0} 
\big\| \left(A_0 - (\lambda +  c^2 / 2 ) \right)^{-1} - \left( - \Delta - \lambda \right)^{-1} M_1 \big\| &\leq \frac{K}{c}, \\
\label{estimate_Phi_z}
\big\| c \Phi_{\lambda + c^2 / 2} M_3 - \Psi_{\lambda} M_2 \big\| &\leq \frac{K}{c}, \\
\label{estimate_Phi_z_star}
\big\| c M_3 \Phi_{\lambda+c^2 / 2}^{\ast} - M_2^{\top} \Psi_{\lambda}^{\ast} \big\| &\leq \frac{K}{c}, \\
\label{estimate_C_z}
\big\| c^2 M_3 \mathcal{C}_{\lambda + c^2 / 2} M_3 - \lambda S(\lambda) M_3 \big\| &\leq \frac{K}{c}, 
\end{align}
\end{subequations}
 are valid for all sufficiently large $c>0$. 
\end{lemma}

\begin{proof}
We use a similar strategy as in the proof of \cite[Proposition~5.2]{BEHL18}. In the following let 
$\lambda \in \dC \setminus \dR$ be fixed. Then $\lambda + \frac{c^2}{2} \in  \dC \setminus \dR$ and 
hence all operators in \eqref{estimate_A_0}--\eqref{estimate_C_z} are well-defined.
One verifies by direct calculation that for sufficiently large $c>0$ and all $t \in [0,1]$
\begin{equation} \label{SqrtBdd}
0 < \frac{1}{2} \left| \sqrt{\lambda} \right| \leq \left| \sqrt{ \lambda + t \frac{\lambda^2}{c^2}} \right| \leq \frac{3}{2} \left| \sqrt{ \lambda} \right| \quad \text{and} \quad \frac{1}{2} \text{Im}  \sqrt{\lambda}  \leq \text{Im} \sqrt{ \lambda + t \frac{\lambda^2}{c^2}} 
\end{equation}
hold. With the well-known asymptotic expansions of the modified Bessel functions and $K_1'(z) = -K_0(z) - \frac{1}{z} K_1(z)$, 
(see \cite{AS64}) one shows that there exist constants $\widehat{K} , \kappa , R > 0$, depending only on $\lambda$, such that
\begin{equation} \label{BesselBounds}
\left| K_j \left( - i \sqrt{  \lambda + t \frac{\lambda^2}{c^2}} |x| \right) \right| \leq \widehat{K} \left\{\begin{array}{ll} |x|^{-1}, & \text{for } |x| < R, \\ e^{- \kappa |x|}, & \text{for } |x| \geq R, \end{array}\right.
\end{equation}
and 
\begin{equation} \label{BesselBoundDerivative}
\left| K_1' \left( - i \sqrt{  \lambda + t \frac{\lambda^2}{c^2}} |x| \right) \right| \leq \widehat{K} \left\{\begin{array}{ll} |x|^{-2}, & \text{for } |x| < R, \\ e^{- \kappa |x|}, & \text{for } |x| \geq R, \end{array}\right.
\end{equation}
hold for all $x \in \dR^2 \setminus \{0\}$, $j \in \{0,1\}$, $t \in [0,1]$, and sufficiently large $c>0$.

Next, with $G_{\lambda+c^2/2}$ defined by \eqref{IntegralkernelfreeDirac}  we find 
\begin{multline} \label{GLambdaNonRelLimit}
G_{\lambda +  c^2  / 2 }(x) = \frac{1}{2\pi c} \sqrt{\lambda+\frac{\lambda^2}{c^2}} K_1 \left( - i \sqrt{\lambda+\frac{\lambda^2}{c^2}} |x| \right) \frac{1}{|x|} ( \sigma \cdot x) \\ + \frac{1}{2 \pi c} K_0 \left( -i \sqrt{\lambda+\frac{\lambda^2}{c^2}}  |x| \right) \left( \frac{\lambda}{c} I_2 + c M_1 \right).
\end{multline} 
Let 
\begin{equation*}
U_\lambda(x) = \frac{1}{2 \pi} K_0 \bigl( - i \sqrt{ \lambda} |x|  \bigr), \quad x \in \dR^2 \setminus \{0\},
\end{equation*}
be the integral kernel of the resolvent of the free Laplace operator; cf. \cite[Chapter~7.5]{T14}. Then 
\begin{equation*}
G_{\lambda + c^2  / 2}(x) - U_\lambda(x) M_1 = t_1(x) + t_2(x) + t_3(x)
\end{equation*}
holds, where the matrix-valued functions $t_1, t_2$, and $t_3$ are given by
\begin{equation*}
\begin{split}
t_1(x) &= \frac{1}{2 \pi c}  \sqrt{ \lambda + \frac{\lambda^2}{c^2}}  K_1 \left( -i \sqrt{ \lambda + \frac{\lambda^2}{c^2}}  |x| \right) \frac{ \sigma \cdot x}{|x|}, \\
t_2(x) &= \frac{1}{2 \pi} \left( K_0 \left( -i \sqrt{ \lambda + \frac{\lambda^2}{c^2}}  |x| \right) -  K_0 \bigl( - i \sqrt{ \lambda} |x| \bigr) \right) M_1, \\
t_3(x) &= \frac{\lambda}{2 \pi c^2} K_0 \left( -i \sqrt{  \lambda + \frac{\lambda^2}{c^2}}  |x| \right)  I_2.
\end{split}
\end{equation*}
With \eqref{SqrtBdd} and \eqref{BesselBounds} applied with $t=1$ one finds that there exist constants $k_1, \kappa, R > 0$, depending only on $\lambda$, such that for  $j \in \{1,3\}$ and sufficiently large $c>0$ one has
\begin{equation*}
\left| t_j(x) \right| \leq \frac{k_1}{c} \left\{\begin{array}{ll}  |x|^{-1}, & \text{for } |x| < R, \\ e^{- \kappa |x|}, & \text{for } |x| \geq R. \end{array}\right. 
\end{equation*}
To estimate $t_2$, we use $K_0' = -K_1$  and obtain with the fundamental theorem of calculus, \eqref{SqrtBdd}, and \eqref{BesselBounds} 
\begin{equation} \label{K0Bound}
\begin{split}
\Bigg| K_0 \left( -i \sqrt{ \lambda + \frac{\lambda^2}{c^2}}  |x| \right) &-  K_0 \left( - i \sqrt{ \lambda} |x| \right)  \Bigg| \leq \int_0^1 \left| \frac{d}{d t} K_0 \left( -i \sqrt{  \lambda + t \frac{\lambda^2}{c^2}}  |x| \right)   \right| d t \\
&= \int_0^1 \frac{|\lambda|^2 |x|}{\left|\sqrt{  \lambda + t \frac{\lambda^2}{c^2}} \right|} \frac{1}{2 c^2}  \left|K_1 \left( -i \sqrt{  \lambda + t \frac{\lambda^2}{c^2}}  |x| \right)    \right| d t \\
&\leq \frac{k_2}{c^2} \left\{\begin{array}{ll}  1, & \text{for } |x| < R, \\ e^{- \frac{\kappa}{2} |x|}, & \text{for } |x| \geq R, \end{array}\right.
\end{split}
\end{equation}
with a constant $k_2$ which depends only on $\lambda$. Thus, if we define $k_3 = 2 k_1 + \frac{k_2 R}{2 \pi}$, then
\begin{equation*}
\left| G_{\lambda + c^2 / 2}(x) - U_\lambda(x) M_1 \right| \leq \frac{k_3}{c} \left\{\begin{array}{ll}  |x|^{-1}, & \text{for } |x| < R, \\ e^{- \frac{\kappa}{2} |x|}, & \text{for } |x| \geq R. \end{array}\right.
\end{equation*}
This estimation for the integral kernel yields with the Schur test; cf. \cite[Proposition~A.3]{BEHL18} for a similar argument,  
\begin{equation*}
\big\| \left(A_0 - (\lambda +  c^2 / 2) \right)^{-1} - \left( -\Delta  - \lambda \right)^{-1} M_1 \big\| \leq \frac{K}{c}
\end{equation*}
for all sufficiently large $c>0$, which is the first claimed estimate~\eqref{estimate_A_0}.

Next, we prove~\eqref{estimate_Phi_z}. Recall that the integral kernel $L_{\lambda}$ of $\Psi_\lambda$ is given by~\eqref{DefIntegralkernPhiLambda}. Using that $\sigma_1 M_3 = M_2$, $\sigma_2 M_3 = -i M_2$, and $M_1 M_3 = 0$, we obtain with \eqref{GLambdaNonRelLimit} the decomposition
\begin{equation*}
c G_{\lambda + c^2 / 2}(x) M_3 - L_{\lambda} (x) M_2 = \tau_1(x) + \tau_2(x) + \tau_3(x)
\end{equation*}
with 
\begin{equation*}
\begin{split}
\tau_1(x) &= \frac{1}{2 \pi} \left( \sqrt{ \lambda + \frac{\lambda^2}{c^2}}   - \sqrt{\lambda} \right) K_1 \left( -i \sqrt{ \lambda + \frac{\lambda^2}{c^2}}  |x| \right) \frac{x_1 - i x_2}{|x|} M_2, \\
\tau_2(x) &= \frac{\sqrt{ \lambda}}{2 \pi} \left( K_1 \left( -i \sqrt{ \lambda + \frac{\lambda^2}{c^2}}  |x| \right) - K_1 \bigl( -i \sqrt{ \lambda}  |x| \bigr) \right) \frac{x_1 - i x_2}{|x|} M_2, \\
\tau_3(x) &= \frac{\lambda}{2 \pi c} K_0 \left( -i \sqrt{ \lambda + \frac{\lambda^2}{c^2}}  |x| \right) M_3.
\end{split}
\end{equation*}
Similar as above it can be shown that there exists a $k_4>0$, depending only on $\lambda$, such that for all $j \in \{1,2,3\}$
\begin{equation*}
\left| \tau_j(x) \right| \leq \frac{k_4}{c} \left\{\begin{array}{ll}  |x|^{-1}, & \text{for } |x| < R, \\ e^{- \frac{\kappa}{2} |x|}, & \text{for } |x| \geq R; \end{array}\right.
\end{equation*}
to see the estimate for $\tau_2$ one has to use \eqref{BesselBoundDerivative}. With the help of the Schur test the 
estimate~\eqref{estimate_Phi_z} follows (see also \cite[Proposition~A.4]{BEHL18} for a similar argument); the constant $k_4$  depends in this case on $\lambda$ and $\Sigma$. The estimate in~\eqref{estimate_Phi_z_star} follows by taking adjoints.

It remains to prove~\eqref{estimate_C_z}. Taking $M_3 \sigma \cdot x M_3 = 0$, which holds for any $x \in \mathbb{R}^2$,  and \eqref{K0Bound} into account we obtain that
\begin{equation*}
\begin{split}
\big|  c^2 M_3 G_{\lambda +  c^2 / 2}(x) M_3 &- \lambda U_\lambda(x) M_3 \big| \\
&= \frac{|\lambda|}{2 \pi} \left|  K_0 \left( -i \sqrt{ \lambda + \frac{\lambda^2}{c^2}}  |x| \right) -  K_0 \bigl( -i \sqrt{  \lambda}  |x| \bigr) \right| \leq \frac{k_5}{c^2}
\end{split}
\end{equation*}
holds for all $x \in \dR^2 \setminus \{0\}$.
Using the dominated convergence theorem, one sees that
\begin{equation*}
\big( c^2 M_3 \mathcal{C}_{\lambda + c^2 / 2} M_3 f \big)(x) = \int_{\Sigma} c^2 M_3 G_{\lambda + c^2 / 2}(x-y) M_3 f(y) d \sigma(y)
\end{equation*}
holds for all $f \in L^2(\Sigma)^2$ and $x \in \Sigma$, i.e. the integral does not have to be understood as principal value.
Thus we obtain with the Schur test \cite[III. Example~2.4]{kato} that
\begin{equation*}
\| c^2 M_3 \mathcal{C}_{\lambda + m c^2} M_3 - \lambda S(\lambda) M_3 \| \leq \frac{K}{c^2}.
\end{equation*}
In this case, the constant $K$ depends on $\lambda$ and $\Sigma$. This yields~\eqref{estimate_C_z} and finishes the proof of this lemma.
\end{proof}

Now we are prepared to prove Theorem \ref{ThmNonRelLimit} and show that $A_{-\alpha c^2/2, \alpha c^2/2}$ converges in the nonrelativistic limit to $T_\alpha$ defined in~\eqref{DefSchroedOpOBT}.

\begin{proof}[Proof of Theorem \ref{ThmNonRelLimit}]
Let  $\lambda \in \dC \setminus \dR$ be fixed. Then, by \cite[Lemma~5.4, Proposition~5.5, Theorem~5.6, and Lemma~5.9]{BHSS22} (see also \cite[Theorem~4.6]{BHOP20}) the operator $I_2 - \alpha c^2   M_3 \mathcal{C}_{\lambda + c^2/2} : L^2(\Sigma)^2 \to L^2(\Sigma)^2$  is boundedly invertible and the resolvent of $A_{- \alpha c^2/2 , \alpha c^2/2} - \frac{c^2}{2}$ is given by
\begin{equation} \label{ResolventDirac}
\begin{split}
\big(A_{-\alpha c^2/2 , \alpha c^2/2} - &(\lambda + c^2/2) \big)^{-1}  = \left( A_{0}-(\lambda +c^2/2)\right)^{-1} \\
&+ \Phi_{\lambda + c^2/2} \left( I - \alpha c^2 M_3 \mathcal{C}_{\lambda + c^2/2} \right)^{-1} \alpha c^2 M_3  \Phi_{\overline{\lambda} + c^2/2}^{\ast}.
\end{split}
\end{equation}
Because of $M_3 = M_3^2$  it follows from \cite[Proposition~2.1.8]{P94} that
\begin{equation*}
\sigma \big( M_3 \mathcal{C}_{\lambda+c^2/2} \big) \cup \{ 0 \} = \sigma \big(  M_3 \mathcal{C}_{\lambda +  c^2 / 2} M_3 \big) \cup \{ 0 \}.
\end{equation*}
In particular, this yields that the operator $I - \alpha c^2 M_3 \mathcal{C}_{\lambda +  c^2/2} M_3$ is boundedly invertible in $L^2(\Sigma)^2$ for all $c>0$ and a direct calculation shows 
\begin{equation} \label{relation_inverse}
(I - \alpha  c^2 M_3 \mathcal{C}_{\lambda + c^2 / 2} )^{-1} M_3 = M_3 ( I - \alpha c^2 M_3 \mathcal{C}_{\lambda + c^2 / 2}  M_3)^{-1}.
\end{equation}
Recall that for $\lambda \in \mathbb{C} \setminus \mathbb{R}$ also $I - \alpha \lambda S(\lambda)$ is boundedly invertible in $L^2(\Sigma)$; cf. Theorem~\ref{ThmPropertiesTAlpha}~(iv).
Hence, we obtain from Lemma \ref{NonrelLimitProp} and \cite[IV. Theorem~1.16]{kato} that
\begin{equation} \label{InverseOperatorConvergence}
\| ( I - \alpha c^2 M_3 \mathcal{C}_{\lambda +  c^2 / 2}  M_3)^{-1} - (I - \alpha \lambda S(\lambda) M_3)^{-1} \| \leq \frac{K}{c}
\end{equation}
holds for all sufficiently large $c>0$ with a constant $K>0$ which depends only on  $\lambda$, $\alpha$, and $\Sigma$. 

To conclude, note that~\eqref{ResolventDirac} and~\eqref{relation_inverse} yield
\begin{equation*}
\begin{split}
\big(A_{-\alpha c^2/2 , \alpha c^2/2}  - &(\lambda + c^2 / 2) \big)^{-1} = \big( A_0 - (\lambda +  c^2 / 2) \big)^{-1} \\
&+ c \Phi_{\lambda+ c^2 / 2} M_3 ( I - \alpha c^2 M_3 \mathcal{C}_{\lambda +  c^2 / 2}  M_3)^{-1} \alpha c M_3 \Phi_{\overline{\lambda}+ c^2/2} ^{\ast},
\end{split}
\end{equation*}
while Theorem \ref{ThmPropertiesTAlpha} (iv) and $M_2 M_3 M_2^{\top} = M_1$ show
\begin{equation*}
\begin{split}
(T_{\alpha} - \lambda )^{-1} M_1 &= ( - \Delta - \lambda)^{-1} M_1 + \Psi_{\lambda} ( I - \alpha \lambda S(\lambda) )^{-1} \alpha \Psi_{\overline{\lambda}}^{\ast} M_1 \\
&=  ( - \Delta - \lambda)^{-1} M_1 + \Psi_{\lambda} M_2 (I - \alpha \lambda S(\lambda) M_3)^{-1}  \alpha M_2^{\top} \Psi_{\overline{\lambda}}^{\ast}.
\end{split}
\end{equation*}
Using Lemma \ref{NonrelLimitProp} and  \eqref{InverseOperatorConvergence} the last two displayed formulae finally lead to the claimed convergence result and 
it also follows that the order of convergence is $\mathcal{O}(\frac{1}{c})$. 
\end{proof}

\appendix

\section{Proof of Propositions~\ref{PropertiesPsi} and~\ref{EigenvalueFuncUnbdd}} \label{Appendix}

Recall that for $\lambda \in \mathbb{C} \setminus [0, \infty)$ the operators $\Psi_\lambda$, $SL(\lambda)$, and $S(\lambda)$ are defined by \eqref{DefPsi},~\eqref{def_SL_potential}, and~\eqref{DefSLBIO}, respectively. 
First, we collect some properties of the single layer potential $SL(\lambda)$  that
are needed in the following. It is well-known that 
$SL(\lambda) : H^{1/2}(\Sigma) \to H^2(\mathbb{R}^2 \setminus \Sigma)$ gives rise to a bounded operator, that $(-\Delta - \lambda) SL(\lambda) \varphi=0$ in $\mathbb{R}^2 \setminus \Sigma$, and that for $\varphi \in H^{1/2}(\Sigma)$ the jump relations
\begin{equation} \label{SL_jump_relation} 
\gamma_D^+ (SL(\lambda) \varphi )_{+} = \gamma_D^- (SL(\lambda) \varphi )_{-}  \quad \text{and} \quad \partial_{\nu} (SL(\lambda) \varphi )_{+}  -   \partial_{\nu}  (SL(\lambda) \varphi )_{-} = \varphi
\end{equation}
hold; cf. \cite{M00} or \cite[Section~3.3]{HU20}. Furthermore, for the single layer boundary integral operator $S(\lambda)$ from \eqref{DefSLBIO} we have $S(\lambda) = \gamma_D SL(\lambda)$ and for all $\varphi \in L^2(\Sigma)$ the representations
\begin{equation} \label{repr_S_SL}
  SL(\lambda) \varphi = (-\Delta - \lambda)^{-1} \gamma_D' \varphi \quad \text{and} \quad S(\lambda) \varphi = \gamma_D (-\Delta - \lambda)^{-1} \gamma_D' \varphi
\end{equation}
hold (see \cite{HU20, M00}); here $\gamma_D:H^1(\mathbb R^2)\rightarrow L^2(\Sigma)$ and 
$\gamma_D':L^2(\Sigma)\rightarrow H^{-1}(\mathbb R^2)$ is the anti-dual operator.

\begin{proof}[Proof of Proposition \ref{PropertiesPsi}]
First, we prove item~(ii). For $\lambda \in \dC \setminus [0, \infty)$ define the operator
\begin{equation} \label{equation_Psi_star}
  \widehat{\Psi}_\lambda := -2i \gamma_D  \partial_{\overline{z}} (-\Delta-\overline{\lambda})^{-1}.
\end{equation}
Since $(-\Delta-\overline{\lambda})^{-1}: L^2(\mathbb{R}^2) \rightarrow H^2(\mathbb{R}^2)$ and $\gamma_D: H^1(\mathbb{R}^2) \rightarrow H^{1/2}(\Sigma)$ are bounded, we get that 
$\widehat{\Psi}_{\lambda} : L^2(\dR^2) \to H^{1/2}(\Sigma)$
is well-defined and bounded. Furthermore, as $H^{1/2}(\Sigma)$ is compactly embedded in $L^2(\Sigma)$ by Rellich's embedding theorem, the 
operator  $\widehat{\Psi}_{\lambda} : L^2(\dR^2) \to L^2(\Sigma)$ is compact.  
Note that $\widehat{\Psi}_\lambda$ is an integral operator with integral kernel  
\begin{equation*}
  \begin{split}
    k(x,y) &= -2i \partial_{\overline{z}} \frac{1}{2 \pi} K_0 \bigl( - i \sqrt{\overline{\lambda}} |x- y| \bigr) \\
    &= \frac{\sqrt{ \overline{\lambda}}}{2 \pi} K_1 \bigl( - i \sqrt{ \overline{\lambda}} |x-y| \bigr) \frac{ x_1-y_1 + i (x_2-y_2)}{|x-y|} \\
    &= \overline{L_\lambda(y-x)},
  \end{split}
\end{equation*}
where we used $K_0'=-K_1$ in the second step and 
$\overline{\sqrt{\lambda}} = - \sqrt{\overline{\lambda}}$ in the last step (recall that $\textup{Im} \sqrt{\omega} > 0$ for $\omega \in \mathbb{C} \setminus [0, \infty)$).
Hence, we conclude that 
\begin{equation*}
  \Psi_\lambda = \widehat{\Psi}_\lambda^*: L^2(\Sigma) \rightarrow L^2(\mathbb{R}^2)
\end{equation*}
is bounded and that all claims in item~(ii) are true.

Next, we show \eqref{relation_Psi_lambda_SL}. Let $\varphi \in L^2(\Sigma)$ and $f \in H^1(\mathbb{R}^2)$. Since $\Delta = 4 \partial_{\overline{z}} \partial_z = 4 \partial_z \partial_{\overline{z}}$, we see that $\partial_{\overline{z}} (-\Delta-\overline{\lambda})^{-1} f =  (-\Delta-\overline{\lambda})^{-1} \partial_{\overline{z}} f$. Hence, item~(ii) and \eqref{repr_S_SL} imply
\begin{equation*}
  \begin{split}
    (\Psi_\lambda \varphi, f)_{L^2(\Sigma)} &= \big(\varphi, -2i \gamma_D  \partial_{\overline{z}} (-\Delta-\overline{\lambda})^{-1} f\big)_{L^2(\mathbb{R}^2)} \\
    &=\big(\varphi, -2i \gamma_D   (-\Delta-\overline{\lambda})^{-1} \partial_{\overline{z}} f\big)_{L^2(\mathbb{R}^2)} \\
    &= \big(-2 i \partial_z SL(\lambda) \varphi, f\big)_{L^2(\mathbb{R}^2)}.
  \end{split}
\end{equation*}
Since $H^1(\mathbb{R}^2)$ is dense in $L^2(\mathbb{R}^2)$, we conclude that~\eqref{relation_Psi_lambda_SL} is true.
In particular, this and the properties of the single layer potential mentioned at the beginning of this appendix imply that 
\begin{equation} \label{mapping_properties_Psi}
  \Psi_\lambda: H^{1/2}(\Sigma) \rightarrow H^1(\mathbb{R}^2 \setminus \Sigma)
\end{equation}
is bounded and for $\varphi \in H^{1/2}(\Sigma)$ we have
\begin{equation} \label{PsiISEV}
i \partial_{ \overline{z}} \left( \Psi_{\lambda} \varphi \right)_{\pm} =  2 \partial_{\overline{z}} \partial_{z}  \left( SL(\lambda) \varphi \right)_{\pm} = \frac{1}{2}  \Delta  \left( SL(\lambda) \varphi \right)_{\pm} = - \frac{\lambda}{2} \left( SL(\lambda) \varphi \right)_{\pm}.
\end{equation}
Since $SL(\lambda) \varphi \in H^1(\dR^2)$, we obtain $ \partial_{ \overline{z}} \left( \Psi_{\lambda} \varphi \right)_{+} \oplus \partial_{ \overline{z}} \left( \Psi_{\lambda} \varphi \right)_{-} \in H^1(\dR^2)$ for any $\varphi \in H^{1/2}(\Sigma)$.

Now, we show (iii). Let $\varphi \in H^{1/2}(\Sigma)$. With~\eqref{PsiISEV} we see that 
\begin{equation*}
  -i ( \gamma_D^+ \partial_{\overline{z}} ( \Psi_{\lambda} \varphi )_{+} +  \gamma_D^- \partial_{\overline{z}} ( \Psi_{\lambda} \varphi )_{-}) = \lambda \gamma_D SL(\lambda) \varphi = \lambda S(\lambda)  \varphi
\end{equation*}
holds. Moreover, we obtain with $SL(\lambda) \varphi \in H^2(\mathbb{R}^2 \setminus \Sigma)$
\begin{equation*}
i ( \nu_1 + i \nu_2) \gamma^{\pm}_D \left( -2 i \partial_z SL(\lambda) \varphi \right)_{\pm} = \partial_\nu \left( SL(\lambda) \varphi \right)_{\pm} - 
i \partial_t \left( SL(\lambda) \varphi \right)_{\pm},
\end{equation*}
where $\partial_t$ 
is the tangential derivative on $\Sigma$. As $SL(\lambda) \varphi \in H^1(\mathbb{R}^2)$, one has the relation $\partial_t\left( SL(\lambda) \varphi \right)_{+} = \partial_t \left( SL(\lambda) \varphi \right)_{-}$ and consequently with~\eqref{SL_jump_relation} 
\begin{equation*}
i ( \nu_1 + i \nu_2) \left( \gamma^{+}_D ( \Psi_{\lambda} \varphi )_{+} -  \gamma^{-}_D ( \Psi_{\lambda} \varphi )_{-} \right) =  \partial_\nu \left( SL(\lambda) \varphi \right)_{+}  -  \partial_\nu \left( SL(\lambda) \varphi \right)_{-} = \varphi.
\end{equation*}
This finishes the proof of (iii).

It remains to prove item~(i). By applying the Wirtinger derivative $\partial_{z}$ to \eqref{PsiISEV} one gets with~\eqref{relation_Psi_lambda_SL} that
\begin{equation*}
-\Delta  \left( \Psi_{\lambda} \varphi \right)_{\pm} = -4 \partial_{z} \partial_{\overline{z}}  \left( \Psi_{\lambda} \varphi \right)_{\pm} = -2 i \lambda \partial_{z}  \left( SL(\lambda) \varphi \right)_{\pm} = \lambda  \left( \Psi_{\lambda} \varphi \right)_{\pm}
\end{equation*} 
holds for all $\varphi \in L^2(\Sigma)$ in the distributional sense.
This,~\eqref{mapping_properties_Psi}, \eqref{PsiISEV}, and the properties of $SL(\lambda)$ described at the beginning of this appendix show that $\Psi_{\lambda} \varphi \in \mathcal{H}_{\lambda}$ for all $\varphi \in H^{1/2}(\Sigma)$ and therefore the mapping $\Psi_{\lambda} : H^{1/2}(\Sigma) \to \mathcal{H}_{\lambda}$ is well-defined. Moreover, it follows from~(iii) that this mapping is injective. To prove that $\Psi_{\lambda} : H^{1/2}(\Sigma) \to \mathcal{H}_{\lambda}$ is surjective, let $f \in \mathcal{H}_{\lambda}$ be fixed. Define the function $\varphi = i (\nu_1 + i \nu_2) \big( \gamma_D^+ f_{+} - \gamma_D^- f_{-} \big) \in H^{1/2}(\Sigma)$ and $g = \Psi_{\lambda} \varphi \in \mathcal{H}_{\lambda}$. By (iii) we have that 
\begin{equation*}
\begin{split}
\gamma_D^+ ( f - g)_{+} - \gamma_D^- ( f - g)_{-} 
&= \gamma_D^+ f_{+} - \gamma_D^- f_{-} +  i (\nu_1 - i \nu_2) \varphi = 0.
\end{split}
\end{equation*}
This shows $f - g \in H^1(\dR^2)$. Moreover, due to $f, g \in \mathcal{H}_{\lambda}$ we have that  $\partial_{\overline{z}} (f - g) \in H^1(\dR^2)$, which implies $f-g \in H^2(\mathbb{R}^2)$. Combining this with $f, g \in \mathcal{H}_\lambda$ we find that $f-g \in \text{ker}\left( -  \Delta - \lambda \right) = \{ 0 \}$, i.e. $f = g = \Psi_\lambda \varphi$. Thus $\Psi_{\lambda} : H^{1/2}(\Sigma) \to \mathcal{H}_{\lambda}$ is also surjective and all claims in assertion~(i) are shown.
\end{proof}

\begin{proof}[Proof of Proposition \ref{EigenvalueFuncUnbdd}]
The proof of item~(i) is divided into 3 steps. In \textit{Step~1} we show that the map $(-\infty, 0) \ni \lambda \mapsto \mu_n(S(\lambda)) \in (0, \infty)$ is continuous and strictly monotonically increasing, in \textit{Step~2} we show that the same is true for the map $(-\infty, 0) \ni \lambda \mapsto \lambda \mu_n(S(\lambda)) \in (-\infty, 0)$. Using these results, we complete the proof of assertion~(i) in \textit{Step~3}.

\textit{Step~1:} Let $n \in \mathbb{N}$. We show that the map $(-\infty, 0) \ni \lambda \mapsto \mu_n(S(\lambda)) \in (0,\infty)$ is continuous and strictly monotonically increasing.
To verify that $\mu_n(S(\lambda))>0$ for $\lambda \in (-\infty,0)$, it suffices to prove that $S(\lambda)$ is a positive self-adjoint operator. From the definition of $S(\lambda)$ in \eqref{DefSLBIO} it follows that $S(\lambda)$ is self-adjoint. Next, let $\varphi \in L^2(\Sigma)$ with $\varphi \neq 0$ and set $f := SL(\lambda) \varphi$. Using the properties of $SL(\lambda)$ described at the beginning of this appendix one finds that $f \neq 0$ and
\begin{equation*}
\begin{split}
\big(S(\lambda) \varphi &, \varphi \big)_{L^2(\Sigma)} = \big( \gamma_D f , \partial_{\nu} f_{+} - \partial_{\nu} f_{-} \big)_{L^2(\Sigma)} \\
&= \big( f_{+} , \Delta f_{+} \big)_{L^2(\Omega_{+})} + \| \nabla f_{+} \|_{L^2(\Omega_{+})}^2 + \big( f_{-} , \Delta f_{-} \big)_{L^2(\Omega_{-})} + \| \nabla f_{-} \|_{L^2(\Omega_{-})}^2 \\
&\geq  \big( f_{+} , \Delta f_{+} \big)_{L^2(\Omega_{+})}  +  \big( f_{-} , \Delta f_{-} \big)_{L^2(\Omega_{-})} = - \lambda \| f \|_{L^2(\dR^2)}^2 > 0.
\end{split}
\end{equation*}
Therefore, $\mu_n(S(\lambda)) > 0$ must be true.

Next, we show that $(-\infty, 0) \ni \lambda \mapsto \mu_n(S(\lambda)) \in (0,\infty)$ is monotonically increasing and continuous. With~\eqref{repr_S_SL} one sees that $S(\cdot): \mathbb{C} \setminus [0, \infty) \rightarrow \mathcal{L}(L^2(\Sigma))$ is holomorphic  and that
$\frac{d}{d \lambda} S(\lambda) = \gamma_D (-\Delta - \lambda)^{-2} \gamma_D'$
holds. 
In particular, for any $\varphi \in L^2(\Sigma)$  the function $(- \infty, 0) \ni \lambda \mapsto (S(\lambda) \varphi , \varphi)_{L^2(\Sigma)}$  is continuously differentiable and
\begin{equation*}
\begin{split}
\frac{d}{d \lambda} \big(S(\lambda) \varphi , \varphi \big)_{L^2(\Sigma)} = \big( (-\Delta - \lambda)^{-1} \gamma_D' \varphi ,  (-\Delta - \lambda)^{-1} \gamma_D' \varphi \big)_{L^2(\Sigma)} &\\
= \| SL(\lambda) \varphi \|_{L^2(\dR^2)}^2 &\geq 0
\end{split}
\end{equation*}
is true.  Thus, the min-max principle implies that the map  $(- \infty, 0) \ni \lambda \mapsto  \mu_n(S(\lambda))$ is monotonically increasing for every $n \in \dN$.
Furthermore, due to the holomorphy of  $S(\cdot): \mathbb{C} \setminus [0, \infty) \rightarrow \mathcal{L}(L^2(\Sigma))$ and the estimate
\begin{equation*}
| \mu_n(S(\eta)) - \mu_n(S(\lambda)) | \leq \| S(\eta) - S(\lambda) \|,\quad \eta,\lambda<0,
\end{equation*}
(see \cite[Satz 3.17]{W00}), we find that $(- \infty, 0) \ni \lambda \mapsto \mu_n(S(\lambda))$ is continuous for $n \in \dN$.  

It remains to show  that the latter map is strictly monotonically increasing.  Define for $\alpha \in \dR \setminus \{0\}$ the operator-valued function $\mathcal{B}_1 : \dC \setminus [0,\infty) \to \mathcal{L}(L^2(\Sigma))$  by $\mathcal{B}_1(\lambda) = I - \alpha S(\lambda)$. By the properties of $S(\lambda)$ it is easy to see that $\mathcal{B}_1$ is holomorphic and $\mathcal{B}_1(\lambda)$ is a Fredholm operator with index $0$ for any fixed $\lambda$, since $S(\lambda)$ is compact in $L^2(\Sigma)$. Moreover, by \cite[Theorem~1.2]{GS14} there exists a constant $K>0$ such that 
\begin{equation} \label{BoundForS}
\| S(\lambda) \| \leq \frac{K}{ \sqrt{2+ |\lambda|}} \ln \sqrt{ 2 + \frac{1}{| \lambda |} } , \quad \lambda \in \dC \setminus [0,\infty).
\end{equation}
Hence, there exists  $\lambda_0 < 0$ such that $\| S(\lambda) \| < |\alpha|^{-1}$  is valid for all $\lambda < \lambda_0$ . This implies that $\mathcal{B}_1(\lambda)$  is boundedly invertible for every  $\lambda < \lambda_0$. Therefore, by \cite[Chapter XI., Corollary~8.4]{GGK90} the set
\begin{equation*}
 \mathcal{M}_{\alpha,1} = \left\{ \lambda \in \dC \setminus [0, \infty) \: \big| \: \mathcal{B}_1(\lambda) = I - \alpha S(\lambda) \text{ is not invertible} \right\}
\end{equation*}
is at most countable and does not have an accumulation point in $\dC \setminus [0 , \infty)$. Now assume that  $\lambda_1 < \lambda_2 < 0$  satisfy $\mu_n(S(\lambda_1)) = \mu_n(S(\lambda_2)) =: \alpha^{-1}$ for some $n \in \dN$.  Then it follows from the monotonicity of  $\lambda \mapsto \mu_n(S(\lambda))$ that $[\lambda_1 , \lambda_2] \subseteq   \mathcal{M}_{\alpha,1} $, which is a contradiction to the fact that $\mathcal{M}_{\alpha,1}$ is at most countable. Therefore, the mapping $(- \infty,0) \ni \lambda \mapsto \mu_n(S(\lambda))$  is continuous and strictly monotonically increasing for $n \in \dN$.

\textit{Step~2:} To show the continuity and strict monotonicity of the map $(- \infty, 0) \ni \lambda \mapsto \lambda \mu_n(S(\lambda))$ for all $n \in \dN$, we note first that the continuity follows from the continuity of the map $\lambda \mapsto \mu_n(S(\lambda))$ shown in \textit{Step~1}.  In order to prove the monotonicity, we use again $\frac{d}{d \lambda} S(\lambda) = \gamma_D (-\Delta - \lambda)^{-2} \gamma_D'$ and compute for a fixed $\varphi \in L^2(\Sigma)$ and $\lambda \in (-\infty, 0)$ with the help of~\eqref{relation_Psi_lambda_SL} and~\eqref{repr_S_SL}
\begin{equation*}
\begin{split}
\frac{d}{d \lambda} \big( \lambda S(\lambda) \varphi , \varphi)_{L^2(\Sigma)} &= \big( S(\lambda) \varphi + \lambda \gamma_D (- \Delta - \lambda)^{-2} \gamma_D' \varphi , \varphi \big)_{L^2(\Sigma)} \\
&= \big( - 4 \gamma_D (-\Delta - \lambda)^{-1}  \partial_{\overline{z}} \partial_{z} (- \Delta - \lambda)^{-1} \gamma_D' \varphi , \varphi \big)_{L^2(\Sigma)} \\
&= \| \Psi_\lambda \varphi \|_{L^2(\dR^2)}^2 \geq 0.
\end{split}
\end{equation*}
Thus, the min-max principle  yields the monotonicity of the mapping $(- \infty, 0) \ni \lambda \mapsto \lambda \mu_n(S(\lambda))$. To see the strict monotonicity, we use a similar strategy as in \textit{Step~1} and define for $\alpha \in \mathbb{R} \setminus \{ 0 \}$ the holomorphic mapping $\mathcal{B}_2 : \dC \setminus [0 , \infty) \to \mathcal{L}( L^2(\Sigma))$ by $\mathcal{B}_2(\lambda) = I - \alpha \lambda S(\lambda)$. Again, $\mathcal{B}_2(\lambda)$ is a Fredholm operator with index zero for any fixed $\lambda$ and it follows from~\eqref{BoundForS} that there exists 
$\lambda_3 < 0$ such that $\| \lambda S(\lambda) \|  < |\alpha|^{-1}$ holds for all $\lambda \in (\lambda_3, 0)$. In particular, $\mathcal{B}_2(\lambda)$ is boundedly invertible  for all $\lambda\in (\lambda_3, 0)$. 
It follows from \cite[Chapter XI., Corollary~8.4]{GGK90} that the set
\begin{equation*}
 \mathcal{M}_{\alpha,2} = \left\{ \lambda \in \dC \setminus [0, \infty) \: \big| \: \mathcal{B}_2(\lambda) = I - \alpha \lambda S(\lambda) \text{ is not invertible} \right\}
\end{equation*} 
 is at most countable and does not have an accumulation point in $\dC \setminus [0 , \infty)$. Now the same argument as in \textit{Step~1} shows that
  $(- \infty, 0) \ni \lambda \mapsto \lambda \mu_n(S(\lambda))$ is strictly monotonously increasing.

\textit{Step~3:} To study the limiting behaviour of $\lambda \mu_n(S(\lambda))$ for $\lambda \rightarrow 0$, note that \eqref{BoundForS} implies $\Vert\lambda S(\lambda)\Vert\rightarrow 0$ for $\lambda\rightarrow 0^-$
and hence,
\begin{equation} \label{Lambda0LimitLSL}
\lim_{\lambda \to 0^{-}} \lambda \mu_n(S(\lambda)) = 0, \quad n \in \dN.
\end{equation}

Next, we consider the limit of $\lambda \mu_n(S(\lambda))$ for $\lambda \to -\infty$. For this purpose, results on Schr\"odinger operators with $\delta$-interactions will be used. Define for $\alpha < 0$ the sesquilinear form 
\begin{equation*}
\mathfrak h_{\delta, \alpha}[f, g] = \big( \nabla f , \nabla g \big)_{L^2(\dR^2)} + \alpha \big( \gamma_D f , \gamma_D g \big)_{L^2(\Sigma)}, \quad f, g \in 
\text{dom}\,\mathfrak h_{\delta, \alpha} = H^1(\dR^2).
\end{equation*}
By \cite{BEKS94, EY01} the form $\mathfrak h_{\delta, \alpha}$ is semi-bounded and closed, and one can show for the self-adjoint operator $H_{\delta, \alpha}$, which is associated with $\mathfrak h_{\delta, \alpha}$ by the first representation theorem, that $\sigma_{\text{ess}}(H_{\delta, \alpha}) = [ 0 , \infty)$, that its discrete spectrum $\sigma_{\text{disc}}(H_{\delta, \alpha})$ is finite, and for $\lambda \in (-\infty,0)$ one has that
\begin{equation} \label{BirmannSchwingerSchroedDelta}
\lambda \in \sigma_\textup{p}( H_{\delta, \alpha}) \quad \Longleftrightarrow \quad -1 \in \sigma_\textup{p}(\alpha S(\lambda));
\end{equation}
see for instance \cite[Lemma~2.3 and Theorem 4.2]{BEKS94} and \cite[Theorems~3.5 and~3.14]{BLL13}. Recall that the eigenvalues 
$\mu_n(S(\lambda))$ are ordered non-increasingly with multiplicities taken into account.
If we order the discrete
eigenvalues of $H_{\delta, \alpha}$ in an increasing way then
the strict monotonicity of $\lambda \mapsto \mu_n(S(\lambda))$ implies that the $k$-th discrete eigenvalue $E_k(\alpha)$ 
(if it exists) satisfies the equation $-1 = \alpha \mu_k(S(E_k(\alpha)))$.

Let $n \in \dN$. Then by  \cite[Theorem~1]{EY01} the operator $H_{\delta, \alpha}$ has at least $n$ negative discrete eigenvalues (counted with multiplicities) 
if $-\alpha>0$ is sufficiently large, and the $n$-th discrete eigenvalue  $E_n(\alpha)$ of $H_{\delta, \alpha}$ admits the asymptotic expansion
\begin{equation} \label{AsymptotikEna}
E_n(\alpha) = - \frac{\alpha^2}{4} + \mu_n(H) + \mathcal{O}(\alpha^{-1} \ln|\alpha|), \qquad \alpha \rightarrow - \infty.
\end{equation}
Here $H$ is a fixed semibounded differential operator on $\Sigma$ that is independent of $\alpha$ and has purely discrete spectrum $\mu_1(H)\leq \mu_2(H)\leq \dots$. Thus for $\alpha \to -\infty$ we obtain with \eqref{BirmannSchwingerSchroedDelta} that
\begin{equation} \label{IneqAsymptotic}
\frac{\alpha}{4} + \frac{|\mu_n(H)|+1}{\alpha} \leq E_n(\alpha) \mu_n(S(E_n(\alpha))) = -\frac{E_n(\alpha)}{\alpha} \leq \frac{\alpha}{4} - \frac{|\mu_n(H)|+1}{\alpha}.
\end{equation}
This shows 
\begin{equation} \label{LambdaInfLimitLSL}
\lim_{\lambda \to -\infty} \lambda \mu_n( S(\lambda) ) = - \infty 
\end{equation}
and finishes the proof of item~(i).

To show item~(ii), we note first that by \eqref{Lambda0LimitLSL}, \eqref{LambdaInfLimitLSL}, and the strict monotonicity and continuity of the mapping  $\lambda \mapsto \lambda \mu_n(S(\lambda))$ it is clear that for any $a<0$ there 
is a unique solution $\lambda_n(a)$ of $\lambda \mu_n(S(\lambda)) = a$. Let $\mu_n(H)$ be as in~\eqref{AsymptotikEna}, 
define the numbers $k_{\pm} = \pm ( |\mu_n(H)|+1 )$ and let 
\begin{equation} \label{Taylor}
\alpha_{\pm} = -2 |a|  \left( \sqrt{1 + \frac{k_{\pm}}{ a^2} }+ 1\right) = -4 |a| - \frac{k_{\pm}}{|a|} + f_{\pm}(a)
\end{equation}
with some functions $f_{\pm}(a) = \mathcal{O}( a^{-3} )$ for large $|a|>0$, where the latter representation holds due to a Taylor series expansion.
Then one has
\begin{equation}\label{jaok}
a = \frac{\alpha_{\pm}}{4} - \frac{k_{\pm}}{\alpha_{\pm}}
\end{equation}
and
it follows with~\eqref{IneqAsymptotic} that
\begin{equation*}
E_n(\alpha_+) \mu_n(S(E_n(\alpha_+)))\leq \frac{\alpha_+}{4} - \frac{|\mu_n(H)|+1}{\alpha_+}=a=\lambda_n(a) \mu_n ( S(\lambda_n(a))
\end{equation*}
and
\begin{equation*}
  \lambda_n(a) \mu_n ( S(\lambda_n(a)) = a = \frac{\alpha_{-}}{4} + \frac{|\mu_n(H)|+1}{\alpha_{-}} \leq E_n(\alpha_-) \mu_n(S(E_n(\alpha_-))).
\end{equation*}
Since $\lambda \mapsto \lambda \mu_n( S(\lambda))$ is monotone we find 
\begin{equation}\label{oho}
E_n( \alpha_{+})\leq \lambda_n(a) \leq E_n( \alpha_{-}). 
\end{equation}
From \eqref{Taylor} we obtain
\begin{equation*}
\frac{1}{4} \alpha_\pm^2=4 a^2 + 2 k_\pm + g_{\pm}(a)
\end{equation*}
with functions $g_{\pm}(a) = \mathcal{O}( a^{-2} )$ for large $|a|>0$ and hence \eqref{AsymptotikEna} implies
\begin{equation}\label{soso3}
\big|E_n(\alpha_\pm) + 4 a^2 + 2 k_\pm + g_{\pm}(a) - \mu_n(H)\big| \leq C_1 \big|\alpha_\pm^{-1} \ln|\alpha_\pm|\big|
\end{equation}
for some constant $C_1 >0$.
Note that there exist constants $C_2, C_3 >0$ such that $C_2 |a| \leq \alpha_\pm \leq C_3 |a|$ holds for large $|a|>0$.
With this we conclude from \eqref{soso3}
that 
\begin{equation} \label{estimate_E_n_alpha_pm}
  \begin{split}
    \big|E_n(\alpha_\pm) + 4 a^2 + 2 k_\pm - \mu_n(H)\big| \leq C_4 \big| a^{-1} \ln|a| \big|
  \end{split}
\end{equation}
holds for some constant $C_4 > 0$ and 
for large $|a|>0$.
Taking \eqref{oho} and~\eqref{estimate_E_n_alpha_pm} into account, one concludes finally that
\begin{equation*}
|\lambda_n(a) + 4a^2| \leq 3 |\mu_n(H)| + 2 + \mathcal{O}(a^{-1} \ln|a|) = \mathcal{O}(1) \quad \text{for } a \to - \infty.
\end{equation*}
\end{proof}


\end{document}